\documentclass[12pt,oneside]{amsart}
\usepackage{amssymb, amsmath, amsthm}

\usepackage{pdfpages}
\usepackage{epstopdf}
\usepackage{graphicx}
\usepackage{color}
\usepackage{pinlabel}
\usepackage{enumerate}

\usepackage[colorlinks = true,
            linkcolor = red,
            urlcolor  = red,
            citecolor = red,
            anchorcolor = red]{hyperref}
\usepackage{amsrefs}
\usepackage[pagewise]{lineno}

\usepackage{pinlabel}

\theoremstyle{plain}
\newtheorem{theorem}{Theorem}[section]

\newtheorem*{theorem*}{Theorem}
\newtheorem*{maintheorem-intro}{Theorem}
\newtheorem*{maintheorem-intro-2}{Theorem~\ref{Bridge number and genus}}
\newtheorem*{theorem-cablingconj}{Theorem~\ref{apps1} (1)}
\newtheorem*{theorem-toroidal}{Specialization of Theorem~\ref{apps1} (2)}
\newtheorem*{theorem-lens}{Theorem~\ref{Bounding distance - special}(1)}
\newtheorem*{theorem-SFS}{Theorem~\ref{Bounding distance - special}(2)}
\newtheorem*{theorem-cosmetic}{Theorem}
\newtheorem*{theorem-bridge}{Specialization of Corollary~\ref{Cor: exceptional bridge}}
\newtheorem*{theorem-Heeggenus}{Corollary~\ref{Cor: exceptional bridge} (2)}

\newtheorem{corollary}[theorem]{Corollary}
\newtheorem{lemma}[theorem]{Lemma}

\theoremstyle{definition}

\theoremstyle{definition}

\newcommand{\R}{{\mathbb R}}

\newcommand{\bi}{\begin{itemize}}
\newcommand{\ei}{\end{itemize}}
\newcommand{\be}{\begin{enumerate}}
\newcommand{\ee}{\end{enumerate}}

\newcommand{\defn}[1]{\emph{#1}}
\newcommand{\bdd}{\partial}
\newcommand{\boundary}{\partial}

\newcommand{\mc}[1]{\mathcal{#1}}

\begin{document}
   % title

   \title[Distortion and Bridge Distance]{Distortion and the Bridge Distance of Knots}
   \author{Ryan Blair}
   \email{ryan.blair@csulb.edu}
   \author{Marion Campisi}
   \email{marion.campisi@sjsu.edu}
   %\author{Jesse Johnson}
   %\email{jjohnson@math.okstate.edu}
   \author{Scott A. Taylor}
   \email{sataylor@colby.edu}
   \author{Maggy Tomova}
   \email{maggy-tomova@uiowa.edu}

   % Note that the short title for running heads goes in square
   % brackets.  This is optional.  The long title goes in curly
   % braces.  In the long title, line breaks are indicated by \\

\begin{abstract}
We extend techniques due to Pardon to show that there is a lower bound on the distortion of a knot in $\mathbb{R}^3$ proportional to the minimum of the bridge distance and the bridge number of the knot. We also exhibit an infinite family of knots for which the minimum of the bridge distance and the bridge number is unbounded and Pardon's lower bound is constant.
\end{abstract}
\maketitle
\date{\today}

\section{Introduction}

The \defn{distortion} of a rectifiable curve $\gamma$ in $\R^3$ was defined by Gromov \cite{Gr78} as:

\[
\delta(\gamma)=\sup\limits_{p,q\in \gamma} \frac{d_{\gamma}(p,q)}{d(p,q)}
\]
where $d_{\gamma}(p,q)$ denotes the shorter of the two arclength distances from $p$ to $q$ along $\gamma$ and $d(p,q)$ denotes the Euclidean distance from $p$ to $q$ in $\mathbb{R}^3$. We can turn this into a knot invariant, by defining for each knot type $K$ in $\mathbb{R}^3$,
\[
\delta(K)=\inf\limits_{\gamma \in K} \delta(\gamma)
\]
where the infimum is taken over all rectifiable curves $\gamma$ representing the knot type $K$.

One of the earliest results giving a lower bound on distortion was due to Gromov who showed that $\delta(\gamma)\geq \frac{\pi}{2}$ with equality if and only if $\gamma$ is the standard round circle \cite{Gr83}. More recently, Denne and Sullivan showed that $\delta(K)\geq \frac{5\pi}{3}$ whenever $K$ is not the unknot \cite{DS}. However, Gromov provided an infinite sequence of knot types with uniformly bounded distortion formed by taking the connected sum of an increasing number of trefoils. It follows that no increasing, unbounded function of crossing number, genus or bridge number can provide a lower bound on the distortion of a knot. 

In view of these results, Gromov  \cite{Gr83} asked if there is a universal upper bound on $\delta(K)$ for all knot types $K$. Pardon \cite{Pardon} answered this question negatively when he showed that the distortion of a knot type is bounded below by a quanitity proportional to a certain topological invariant, called representativity, which is known to not have a universal upper bound. Subsequently, Gromov and Guth \cite{GG}  also answered Gromov's question in the negative by giving a lower bound on $\delta(K)$ in terms of the hyperbolic volume of certain covers of $S^3$ ramified over $K$.

In this paper, we use some of the ideas from Pardon's proof to show that distortion is bounded below by a quantity involving two topological knot invariants: bridge number and bridge distance.

\begin{theorem} \label{thm:main}
Let $K$ be a knot type in $\mathbb{R}^3$. Then
\[
\delta(K)\geq \frac{1}{160}\min(d(K),2b(K))
\]
where $d$ is bridge distance and $b$ is bridge number.
\end{theorem}

Furthermore, we show that our bound is arbitrarily stronger than Pardon's by providing a family of knots, based on a family of knots constructed by Johnson and Moriah in \cite{JM}, for which representativity is constant while $\min(d(K),2b(K))$ is unbounded. In addition, we give an upper bound for the distortion of this family of knots which is also in terms of bridge number and bridge distance.

This paper is structured as follows: Section \ref{sec:bridge} contains relevant definitions, including those of bridge number and bridge distance. Theorem \ref{thm:main} is proved in Section \ref{sec:proof}. In Section \ref{sec:altJMknots} we explain Johnson and Moriah's construction and its relevance to our lower bounds. In Section \ref{thm:upperbound} we give an \emph{upper bound} for the distortion of the Johnson-Moriah knots proportional to the product of bridge distance with the square of the bridge number.

\section{Preliminaries}\label{sec:bridge}

\subsection{Knots in $\R^3$ and $S^3$:} As is usual in knot theory, we will freely switch between considering knots in $\R^3$ and $S^3$. Two knots in $\R^3$  (or $S^3$) are equivalent if there is an orientation-preserving self-homeomorphism of $\R^3$ (or $S^3$, respectively) taking one knot to the other. Since distortion is a metric quantity, we will make this precise by choosing a point at infinity in $S^3$ and considering stereographic projection from $S^3$ to $\R^3$. As is well-known, two knots in $S^3$ are equivalent if and only if their images in $\R^3$ are equivalent.

\subsection{Tangles:}

Suppose $N$ is a 3-manifold containing a properly embedded, possibly disconnected, 1-manifold $\tau$. If a properly embedded surface $F \subset N$ is transverse to $\tau$, we will write $F \subset (N,\tau)$ and consider the points $\tau \cap F$ to be \emph{punctures} on $F$. Two properly embedded punctured surfaces in $(N,\tau)$ are \emph{equivalent} if they are transversely properly isotopic with respect to $\tau$. A curve $\lambda \subset F$ will be called \defn{essential} if $\lambda$ is disjoint from $\tau$, is not boundary parallel, and doesn't bound a disk in $F$ with fewer than two punctures. The surface $F$ will be said to be \defn{incompressible} in $(N,\tau)$ if there is no disk $D$ (called a \defn{compressing disk}) disjoint from $\tau$ with $D \cap F=\bdd D$ an essential curve in $F$. If $F$ is a sphere, it is called an \defn{inessential sphere} if it bounds a ball disjoint from $\tau$ or bounds a ball containing a single, boundary parallel subarc of $\tau$. A connected incompressible surface $F$ will be called \defn{essential} if it is not an inessential sphere and if there is no parallelism relative to $\tau$ between $F$ and some collection of components of $\bdd N$. A disconnected surface is essential if every component of the surface is essential. If $F$ is a closed, connected, punctured separating surface in $N$ then $F$ is \emph{bicompressible} if there exists a compressing disk for $F$ contained to each side.

 A \defn{punctured 3-sphere} is the result of removing finitely many open balls from $S^3$. A \emph{tangle} $(B, \tau)$ consists of a punctured 3-sphere $B$ and a properly embedded 1--manifold $\tau$ that has no closed components. A tangle $(B, \tau)$ is \emph{irreducible} if every unpunctured 2-sphere in $(B, \tau)$ bounds a ball disjoint from $\tau$. The tangle $(B, \tau)$ is \emph{prime} if every twice punctured sphere is either inessential or is parallel to a component of $\bdd B$ relative to $\tau$. The tangle $(B, \tau)$ is \emph{trivial} if there is a component $P \subset \bdd B$ so that no arc of $\tau$ has both of its endpoints on the same component of $\bdd B \setminus P$ and there is a collection of disjoint compressing disks for $P$ contained in the complement of $\tau$ so that reducing $B$ along this collection results in components that are either balls containing a single boundary parallel arc or are homeomorphic to $S^2 \times I$ possibly containing vertical arcs. The component $P$ is called the positive boundary to $B$, we write $P=\bdd_+B$.

Suppose $B$ is a ball and $(B, \tau)$ is a trivial tangle. A \defn{spine} for $(B, \tau)$ is a graph $\Gamma$ with a single vertex $v$ of degree $n$ and edges that connect $v$ to each of the components of $\tau$ so that the boundary of a neighborhood of $\Gamma$ is isotopic to $\partial
 B$ relative to $\tau$. Spines can be defined for tangles in other manifolds as well but this more general definition is not needed for the current paper.
 
\subsection{Bridge spheres:} An $n$-bridge sphere $\Sigma$ for a PL knot $\gamma$ in $S^3$ is a sphere transverse to $\gamma$, dividing $(S^3,\gamma)$ into two trivial $n$-component tangles. The \defn{bridge number} of a knot type $K$ in $S^3$ is the minimum $n$ such that there is an $n$-bridge sphere for a PL representative of $K$. The bridge number of $K$ is denoted $b(K)$. 

 A \defn{bridge sphere} $\Sigma$ for a tangle $(B, \tau)$ is a 2--sphere in $B$ transverse to $\tau$ such that $\Sigma$ divides $(B,\tau)$ into two trivial tangles, each having $\Sigma$ as its positive boundary. 
 
 %If in each trivial tangle the remnants of $\tau$ have $n$-components, we say that $\Sigma$ is an \defn{$n$-bridge sphere}. 

\subsection{Distance:} We will utilize ``sufficiently complicated" bridge spheres, where complexity is measured via distances between disk sets in the curve complex.  Let $\gamma$ be a PL knot in $S^3$ and let $\Sigma$ be any bridge sphere for $\gamma$ separating $S^3$ into balls $B_1$ and $B_2$. Define the (1-skeleton of the) \defn{curve complex} $\mathcal{C}(\Sigma)$ to be a graph whose vertices are isotopy classes of essential simple closed curves. In the context of the curve complex, we will always discuss curves up to isotopy even if we don't explicitly state so. Two vertices are connected by an edge if their corresponding curves may be realized disjointly. The vertex set of the curve complex $\mathcal{C}(\Sigma)$ has a natural metric constructed by assigning each edge length one and defining the distance between two vertices to be the length of the shortest path between them.
Let the \emph{disk set} $\mathcal{D}_i \subset \mathcal{C}(\Sigma)$ be the set of vertices which correspond to essential curves in $\Sigma$ that bound compressing disks in $B_i$, and the \emph{distance of $\Sigma$}, denoted $d(\Sigma)$, to be

\[ d(\Sigma) = \min\{ d(c_1,c_2): c_i \in \mathcal{D}_i\}.\] If $\Sigma$ is punctured four or fewer times (i.e., when the curve complex is empty or disconnected), we define the distance of $\Sigma$ to be infinite. The \defn{distance of $K$}, $d(K)$, is the maximum possible distance $d(\Sigma)$ of any bridge sphere $\Sigma$ for a PL representative of $K$ that also realizes the bridge number of $K$. It follows from \cite{To07}, that as long as $b(K) \geq 3$, this maximum is a well-defined positive integer.

The utility of this definition is made clear by the following two theorems of Johnson and Tomova. The first is a version of Theorem 4.4 from \cite{johntom} and the second is implicit in the proof of Theorem 4.2 from \cite{johntom}. We have restated the theorems to be in the language of punctured surfaces. A \defn{c-disk} for a surface $S \subset (M, \tau)$ is a disk with boundary an essential curve in the punctured surface $S$; transverse to $\tau$; with interior disjoint from $S$ and intersecting $\tau$ in at most 1 point. 

For the statements below, suppose that $(N, \kappa)$ is a tangle and that $M \subset N$ is an embedded punctured 3-sphere with $\boundary M$ transverse to $\kappa$. Let $T = M \cap \kappa$. Assume that every compressing disk for $\boundary M \subset (N, T)$ lies interior to $M$. Let $\Sigma$ be a bridge sphere for $(M, \kappa)$.

\begin{theorem}\label{tangdist1}(Johnson and Tomova)
Let $\Sigma'$ be a bridge sphere for $(N,\kappa)$.  Then one of the following holds:
{\bi
\item After isotopy and surgery of $\Sigma' \cap M$ along c-disks in $M$, we obtain a compressed surface $\Sigma''$ such that $\Sigma'' \cap M$ is parallel to $\Sigma$. In particular, $|\Sigma\cap \kappa| \leq |\Sigma '\cap \kappa|$.
\item $d(\Sigma) \leq |\Sigma'\cap \kappa|$.
\item $|\Sigma \cap \kappa|\leq 5$.
\ei}
\end{theorem}

\begin{theorem}\label{tangdist2}(Johnson and Tomova)
Let $F \subset (N, \kappa)$ be a separating planar punctured surface transverse to $\boundary M$.  Additionally, suppose $\boundary (F \cap M)$ is essential in the punctured surface $\boundary M$. Then one of the following holds:
{\bi
\item $d(\Sigma) \leq |F\cap \kappa|+|\bdd F|$.
\item $ |\Sigma \cap \kappa|\leq 5$.
\item Each component of $F \cap (M,T)$ is boundary parallel in $(M,T)$
\ei}
\end{theorem}

\subsection{PL assumption:} Given a rectifiable simple closed curve $\beta^* \subset \R^3$, there is an equivalent PL simple closed curve $\beta \subset \R^3$. The preimage of $\beta$ under stereographic projection is equivalent to a PL knot $\beta'$ in $S^3$. Suppose that $\Sigma$ is a bridge sphere for $\beta'$, disjoint from the point at infinity. The composition of self-homeomorphisms of $S^3$ and $\R^3$ with the stereographic projection takes $\Sigma$ to a sphere in $\R^3$ which intersects $\beta^*$. This sphere may not be PL. By judiciously choosing the self-homeomorphisms and approximating stereographic projection with a PL map, it is possible (see \cite{Pardon}) to ensure that the non-PL points occur only near $\beta^*$. We will continue to refer to this sphere as a bridge sphere for $\beta^*$ and omit the details required for handling the case when $\beta^*$ is not PL. These details are the same as in Pardon's paper. Henceforth, for simplicity, we will assume that all knots, tangles, and surfaces are PL and will freely pass back and forth between $S^3$ and $\R^3$. In particular, $\beta^* = \beta$ except where indicated. Our main theorem applies to non-PL, but rectifiable, knots by applying the same methods as Pardon. Where possible, we will adopt Pardon's notation (e.g. using $\beta^*$) to make the comparison with his paper easier.

\section{Proof of Theorem \ref{thm:main}}\label{sec:proof}

The initial steps of our proof are modelled on Pardon's. Suppose that $K$ is a knot type in $\R^3$ and that $\beta, \beta^*$ are representatives of $K$.  Suppose that $\beta$ is PL and has a bridge sphere $\Sigma$ and that $\beta^* \subset \R^3$ is rectifiable. As we mentioned previously, for ease of exposition, we will assume that $\beta^*$ is actually PL, in which case we can consider $\Sigma$ as a PL bridge sphere for $\beta^*$. The general case can be handled as in Pardon's paper.

We will show:
\[
\delta(\beta^*) \geq \frac{1}{160}\min(d(\Sigma), |\Sigma \cap \beta|).
\]
When applied to a minimal bridge sphere $\Sigma$ of maximum possible distance, Theorem \ref{thm:main} follows.

To obtain a contradiction, let $k = \min(d(\Sigma), |\Sigma \cap \beta^*|)$ and assume that $\delta(\beta^*) < k/160$. By Gromov's result, $\delta(K)\geq \frac{\pi}{2}$.  Hence, $k>80\pi$, and in particular the curve complex for $\Sigma$ is connected. By \cite{BS},  the distance of a bridge sphere for a composite knot is at most two and therefore $(S^3, \beta^*)$ is prime.

As in Pardon's paper, let $Box(r)$ denote the equivalence class of
\[
\big\{(x,y,z) \in \R^3 : |x| \leq r, |y| \leq 2^{1/3}r, |z| \leq 2^{2/3} r\big\}
\]
up to Euclidean isometry. Departing from \cite{Pardon}, define  $\mc{R}\subset\mathbb{R}$ by declaring $r \in \mc{R}$ if and only if there exists a representative of $Box(r)$ containing a spine for one of the complementary tangles to $\Sigma$.  Modelling our work on Pardon's, we will show that there exists a positive real number $\Delta$ such that $(1-\Delta)\mc{R}\subset \mc{R}$. This leads to a contradiction as a spine for a bridge sphere of a non-trivial tangle cannot be isotoped relative to the knot to lie in an arbitrarily small region of $\mathbb{R}^3$

Since distortion is invariant under rigid motions and scaling, we can assume that $1\in \mc{R}$ and that the corresponding representative of $Box(1)$ is standardly positioned and centered at the origin. 

As in \cite{Pardon}, for $\epsilon \leq 1/7$, there exists $r_1\in (1,1+\epsilon)$ such that
\[
|\beta^*\cap \partial Box(r_1)|\leq 10(1+\frac{1}{\epsilon})\delta(\beta^*).
\]

Additionally,  there exists $s_1\in(-\epsilon, \epsilon)$ such that

$$|\beta^*\cap (Box(r_1)\cap \{z=s_1\})|\leq 5(1+\frac{1}{\epsilon})\delta(\beta^*).$$

As $\epsilon \leq 1/7$, from the assumption, it follows that $20(1+\epsilon^{-1})\delta(\beta^*)<k$.

Hence,
\[\begin{array}{rcl}
|\beta^*\cap \partial Box(r_1)|&\leq& k/2, \text{ and } \\
|\beta^*\cap (Box(r_1)\cap \{z=s_1\})|&\leq& k/4.
\end{array}
\]

Let $S=\partial Box(r_1)$, $D=(Box(r_1)\cap \{z=s_1\})$ and $\Theta=S\cup D$. Note that
\[
|S\cap \beta^*|+2|D\cap \beta^*|\leq 20(1+\frac{1}{\epsilon})\delta(\beta^*)<k.\]

By applying Pardon's techniques, we may assume that $\beta^*$ is transverse to $\Theta$ and disjoint from $S\cap D$. There exists $\zeta$ sufficiently small so that a closed regular $\zeta$-neighborhood of $\Theta$, denoted by $N$, meets $\beta^*$ in $|S\cap \beta^*|+|D\cap \beta^*|$ arcs. The complement of $N$ in $\mathbb{R}^3$ consists of three regions. We will denote the closures of the two bounded regions by $H_1^*$ and $H_2^*$ while the closure of the third region will be denoted by $B^*$, see Figure \ref{fig:Theta}.

\begin{figure}[ht]
\labellist \small\hair 2pt
\pinlabel {$N$} [b] at 345 20
\pinlabel {$H_1^*$} [b] at 210 110
\pinlabel {$H_2^*$} [b] at 210 285
\pinlabel {$B^*$} [b] at 30 15
\pinlabel {$\Theta$} [b] at 340 205
\endlabellist
\includegraphics[scale=.5]{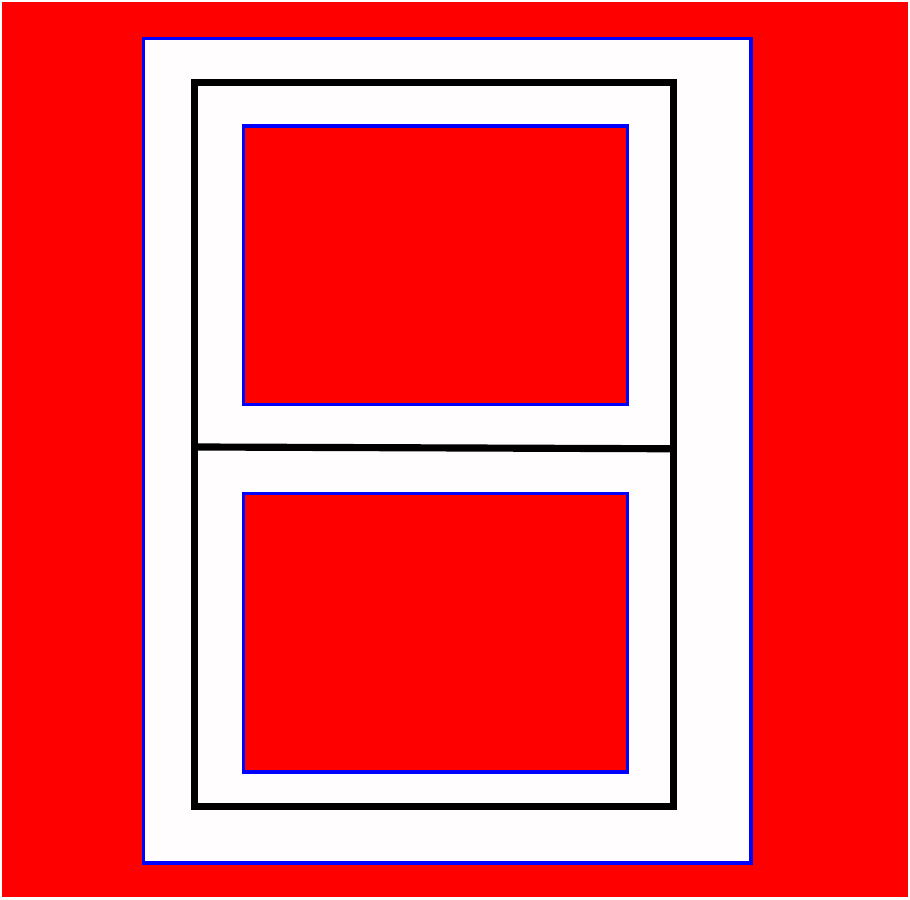}
\caption{A schematic picture of $N$, $H_1^*$, $H_2^*$, $B^*$ and $\Theta$.}
\label{fig:Theta}
\end{figure}
%\scott{Change $M$ to $N$ and use pinlabel for the labels.}

Since $1\in \mc{R}$, then $r_1\in \mc{R}$ and there is a spine $\Gamma$ for one of the tangles complementary to $\Sigma$ contained in the interior of the 3-ball bounded by $\partial B^*$. We may assume that $\Gamma$ has been isotoped to be transverse to $\partial H_1^*$ and $\partial H_2^*$.

\begin{lemma}\label{first isotopy}
$\Gamma$ can be isotoped so that it is contained in one of $N$, $H_1^*$, or $H_2^*$.
\end{lemma}

\begin{proof}

We will construct an isotopy of $\boundary N$ relative to $\beta^*$ so that after the isotopy $\Gamma$ is disjoint from $\partial N$ and continues to be disjoint from $B^*$. Changing our perspective so that it is $\Gamma$ which is subject to the isotopy, gives us the desired result.

Let $\Sigma_\Gamma$ be the boundary of a small regular neighborhood of $\Gamma$. Then $\Sigma_\Gamma$ is isotopic to $\Sigma$ relative to the knot and therefore separates $(S^3, \beta^*)$ into two trivial tangles $(B_1,T_1)$ and $(B_2,T_2)$. The sphere $\partial B^*$ is contained in one of these balls, say $(B_1,T_1)$. 

Let $D_1$ be a compressing disk for $\partial B^*$ in $(B_1,T_1)$. This disk may intersect $\boundary H_1^* \cup \boundary H_2^*$ in simple closed curves. Starting with an innermost such curve on $D_1$ bounding a disk $\tilde{D} \subset D_1$ compress $\boundary H_1^* \cup \boundary H_2^*$ along $\tilde{D} $. Repeat this process with a new innermost curve of intersection until $D_1$ is disjoint from the partially compressed $\boundary H_1^* \cup \boundary H_2^*$. Finally, compress $\partial B^*$ using $D_1$. If the compressed $\partial B^*$ has other compressing disks in $(B_1, T_1)$ repeat the series of compressions with another compressing disk $D_2$. This process terminates with surfaces $\partial B^c \cup \partial H^{pc}_1 \cup \partial H^{pc}_2$ where $\partial B^c$ is the result of compressing $\partial B^*$ completely along disks $D_1, D_2, ...,D_k$ and $\partial H^{pc} \cup \partial H^{pc}$ is the result of partially compressing $\boundary H_1^* \cup \boundary H_2^*$ along sub disks of these disks.

As $\partial B^c$ is an incompressible surface in a trivial tangle, it is the union of inessential twice punctured spheres. Let $\mc{U}\subset B_1$ be the union of the 3-balls containing unknotted arcs which are bounded by $\partial B^c$. There is an isotopy of $\partial B^c \cup \partial H^{pc}_1 \cup \partial H^{pc}_2$ in $B_1$ after which $\partial B^c$ and any component of $\partial H^{pc}_1 \cup \partial H^{pc}_2$ contained in $\mc{U}$ is contained in an arbitrarily small neighborhood of $\beta^*$, denoted $\eta(\beta^*)$. Since the image of this isotopy is contained in $B_1$, $\partial B^c$ remains disjoint from $\Gamma$ during the course of this isotopy. 

 Note that we can recreate an isotopic copy $F^{n}$ of $\boundary N$ by tubing together the components of $\partial B^c \cup \partial H^{pc}_1 \cup \partial H^{pc}_2$ along a series of (possibly nested) annuli contained in the boundary of a regular neighborhood of the core of each surgery disk. Moreover, the isotopy taking $\boundary N$ to $F^{n}$ can be chosen so that its image is contained in $B_1$. Hence, the isotopy from $\boundary N$ to $F^{n}$ restricts to an isotopy of $\partial B^*$ that is disjoint from $\Gamma$. Moreover $\mc{U}$ is also disjoint from $\Gamma$. See Figure \ref{fig:Iso1}.

\begin{figure}[ht]
\labellist \small\hair 2pt
\pinlabel {$\beta^*$} [b] at 525 12
\pinlabel {$\eta(\beta^*)$} [b] at 530 180
\pinlabel {$\Gamma$} [b] at 238 460
\pinlabel {$F^n$} [b] at 234 260
\pinlabel {$\partial N$} [b] at 253 355
\endlabellist
\includegraphics[scale=.6]{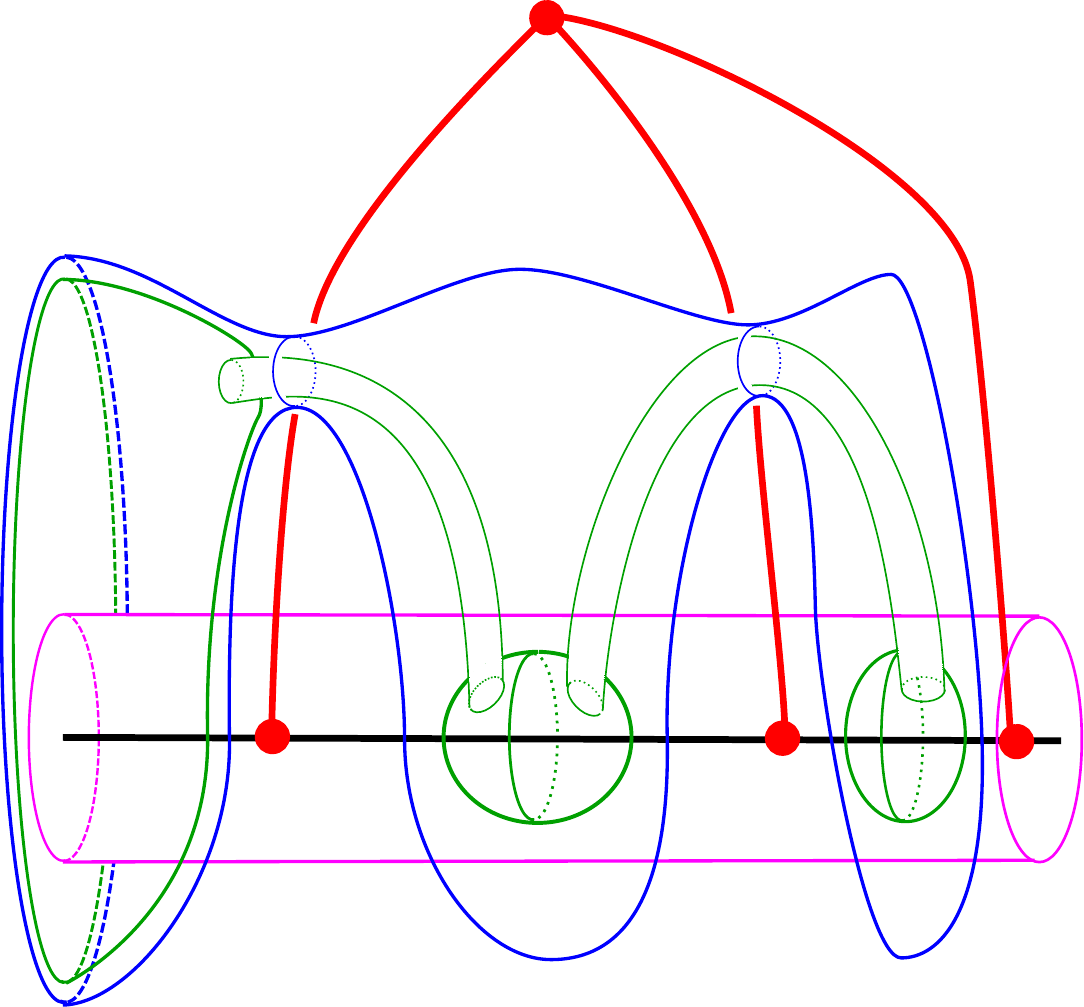}
\caption{A possible configuration of $F^n$ and $\boundary N$.}
\label{fig:Iso1}
\end{figure}

Next, maximally compress $\partial B^c \cup \partial H^{pc}_1 \cup \partial H^{pc}_2$ in the complement of $\beta^*$ using disks $D_{k+1},...,D_r$ and note that these compressions can be taken to be disjoint from $\partial B^c$. If any component of the compressed surface is not an inessential  2-sphere, then we have produced an essential sphere $F$ transverse to $\beta^*$ such that $|F \cap \beta^*|<k$. By Theorem \ref{tangdist2}, $d(\Sigma)\leq |F \cap \beta^*|<k=\min(d(\Sigma), |\Sigma \cap \beta^*|)$, a contradiction. Hence, the result of maximally compressing $\partial B^c \cup \partial H^{pc}_1 \cup \partial H^{pc}_2$ in the complement of $\beta^*$ is a collection of inessential spheres denoted by $\partial B^c \cup \partial H^{c}_1 \cup \partial H^{c}_2$. After an isotopy, we can assume that $\partial B^c \cup \partial H^{c}_1 \cup \partial H^{c}_2$ is contained in a regular neighborhood of a single point on $\beta^*$. Moreover, since before this isotopy $\partial B^c$ was disjoint from $\Gamma$, although $\partial B^c$ may intersect $\Gamma$ during this isotopy, $\Gamma$ is contained in the same component of $S^3 \setminus  \partial B^c$ both before and after this isotopy. 

We can now undo all compressions described above in reverse order. First, we will tube $\partial B^c \cup \partial H^{c}_1 \cup \partial H^{c}_2$ along the boundaries of regular neighborhoods of arcs dual to the disks $D_{k+1},...,D_r$ to form a surface $F^c$ isotopic to $\partial B^c \cup \partial H^{pc}_1 \cup \partial H^{pc}_2$. This surface lies in the neighborhood of a 1-complex in $\mathbb{R}^3$ with a single vertex on $\beta^*$ and thus can be isotoped to be disjoint from $\Gamma$. Then we will undo the compressions that we did along $D_{1},...,D_k$ and their sub-disks as described above obtaining a surface $F^n$ that is isotopic to $\bdd N$ but which is disjoint from $\Gamma$ and such that $\Gamma$ is contained in the same component of $S^3\setminus \partial B^*$ as $N$. Now, instead of viewing the isotopy described as an isotopy of $\bdd N$ we view it as an isotopy of $\Gamma$ after which $\Gamma$ is contained in one of $N$, $\boundary H_1^*$, or $\boundary H_2^*$.

\end{proof}

\begin{lemma}
After the isotopy described in Lemma \ref{first isotopy} $\Gamma$ is contained in one of $H_1^*$ or $H_2^*$.
\end{lemma}

\begin{proof}

 By Lemma \ref{first isotopy}, it suffices to show that $\Gamma$ is not contained in $N$. We will proceed by contradiction and thus assume $\Gamma \subset N$.

Let $\Sigma_{\Gamma}$ be the boundary of a closed regular neighborhood of $\Gamma$ in $N$ and note that $\Sigma_{\Gamma}$ is transversely isotopic to $\Sigma$ relative to $\beta^*$ and it is therefore compressible on the side containing $\Gamma$. We will first construct a new manifold $N_c$ by adding 2-handles to $N$ so that $\Sigma_{\Gamma}$ is bicompressible in $N_c$. If $\Sigma_{\Gamma}$ is already bicompressible in $N$ let $N_c=N$. Otherwise, let $E$ be a compressing disk for $\Sigma_{\Gamma}$ in $(S^3,\beta^*)$ to the side disjoint from $\Gamma$. Assume that we have isotoped $E$ so that $|E\cap \partial N|$ is minimal. In particular, all curves of $E\cap \partial N$ are essential in $\partial N$. Let $E_N$ be the planar surface embedded in $N$ that is the closure of the complementary component of $E\cap \partial N$ in $E$ that contains $\partial E$. Hence, $\partial E_N$ is the union of $\partial E$ and $\alpha_1,..., \alpha_n$, a collection of disjoint essential curves in $\partial N$. Create the new 3-manifold $N_c$ by attaching 2-handles $B_1, ..., B_k$ to $N$ along each distinct isotopy class of $\alpha_1,..., \alpha_n$. By construction, $\Sigma_{\Gamma}$ is bicompressible in $(N_c, \beta^* \cap N_c)$. In what is to follow we will always allow the collection of 2-handles $B_1, ..., B_k$ to be empty.

\textbf{Claim:} The distance of $\Sigma_{\Gamma}$ as a bicompressible surface in $(N_c, \beta^* \cap N_c)$ is greater than or equal to the distance of $\Sigma_{\Gamma}$ as a bicompressible surface in $(S^3,\beta^*)$.

\begin{proof}[Proof of Claim]

Let $\omega$ be an essential curve in $\Sigma_{\Gamma}$ that bounds a compressing disk $E_{\omega}$ in $N_c$ to the side of $\Sigma_{\Gamma}$ that is disjoint from $\Gamma$. After an isotopy of $E_{\omega}$, we can assume that $E_{\omega}$ meets each of $B_1, ..., B_k$ in a collection of disks each of which is parallel to $D^2\times \{\frac{1}{2}\}$ in $D^2\times I=B_i$ for some $i$. We use the following procedure to construct an immersed compressing disk for $\Sigma_{\Gamma}$ in $(S^3,\beta^*)$ with boundary $\omega$. Let $\gamma$ be a boundary component of the closure of $E_{\omega}\setminus (B_1\cup ... \cup B_k)$ other than $\omega$. By construction, $\gamma$ bounds a disk in $E_{\omega}$ and is isotopic in $\partial N$ to $\alpha_j$ for some $j$. Glue a copy of the disk $\alpha_j$ bounds in $E$ to $\gamma$ and repeat this process for each boundary component of $E_{\omega}\setminus (B_1\cup ... \cup B_k)$ other than $\omega$. The result is an immersed compressing disk for $\Sigma_{\Gamma}$ in $(S^3,\beta^*)$ with boundary $\omega$ such that the singularity set is contained in the interior of the disk. By the loop theorem, this implies that $\omega$ bounds a compressing disk for $\Sigma_{\Gamma}$ in $(S^3,\beta^*)$ to the side of $\Sigma_{\Gamma}$ that is disjoint from $\Gamma$. Since $\omega$ was arbitrary, every curve in $\Sigma_{\Gamma}$ that bounds a compressing disk for $\Sigma_{\Gamma}$ in $N_c$ to the side disjoint from $\Gamma$ also bounds a compressing disk for $\Sigma_{\Gamma}$ in $(S^3,\beta^*)$ to the side disjoint from $\Gamma$. Thus, the distance of $\Sigma_{\Gamma}$ as a bicompressible surface in $(N_c, \beta^*\cap N_c)$ is greater than or equal to the distance of $\Sigma_{\Gamma}$ as a bicompressible surface in the $(S^3,\beta^*)$.
\end{proof}

Maximally compress $\Sigma_{\Gamma}$ in $N_c$ to the side opposite of $\Gamma$ and ignore any inessential spheres to obtain a surface $\Sigma_{\Gamma}^*$.

\begin{figure}[ht]
\labellist \small\hair 2pt
\pinlabel {$N$} [b] at 90 390
\pinlabel {$M$} [b] at 280 285
\pinlabel {$\Gamma$} [b] at 303 213
\pinlabel {$\Sigma_{\Gamma}^{*}$} [b] at 240 212
\endlabellist
\includegraphics[scale=.6]{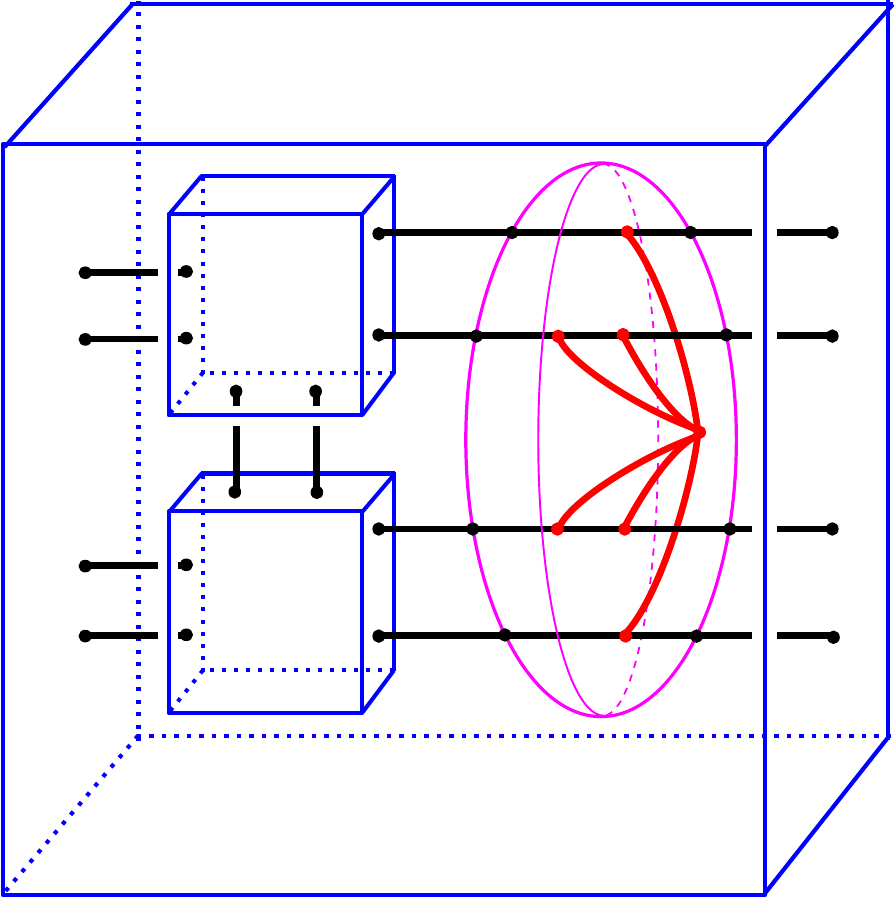}
\caption{An example of $M$, $\Sigma_{\Gamma}^*$ and $\Gamma$ in $N$.}
\label{fig:MandN}
\end{figure}

\textbf{Case 1:} $\Sigma_{\Gamma}^*$ is non-empty and is compressible to the side containing $\Gamma$.
Let $M$ denote the submanifold of $N_c$ that $\Sigma_{\Gamma}^*$ bounds and contains $\Gamma$. See Figure \ref{fig:MandN} which depicts the case when $N=N_c$.

This compressing disk will play the role of $F$ in  the statement of Theorem \ref{tangdist2} and we can conclude that the distance of $\Sigma_{\Gamma}$ as a bicompressible surface in $(M, \beta^*\cap M)$ is less than or equal to $|\bdd F|=1$. This is a contradiction since the distance of $\Sigma_{\Gamma}$ as a bicompressible surface in $(M, \beta^*\cap M)$ is greater than or equal to the distance of $\Sigma$ as a bicompressible surface in the complement of $\beta^*$ which is greater than or equal to $k$ and $k > 80\pi$.

\textbf{Case 2:} $\Sigma_{\Gamma}^*$ is non-empty and is not compressible to the side containing $\Gamma$.

Again, let $M$ denote the submanifold of $N_c$ that $\Sigma_{\Gamma}^*$ bounds and contains $\Gamma$. In this case, $\Sigma_{\Gamma}^*$ is incompressible in $(N_c,\beta^*\cap N_c)$. Let $\Sigma_N$ be the bridge sphere for the tangle $(N, \beta^*\cap N)$, which is constructed by tubing a surface isotopic to $\partial H_1^*$ to a surface isotopic to $\partial H_2^*$. By choosing a tube disjoint from $\beta^*$, we can assume $\Sigma_N$ meets $\beta^*$ in $|S\cap \beta^*|+2|D\cap \beta^*|$ points. By Theorem \ref{tangdist1}, we can compare the bridge sphere $\Sigma_N$ for $(N_c,\beta^*\cap N_c)$ and the bridge sphere $\Sigma_{\Gamma}$ for $(M,\beta^*\cap M)$ and conclude that one of the following holds:

\begin{enumerate}

\item $|\beta^*\cap \Sigma_N|\geq |\beta^*\cap \Sigma_{\Gamma}|$.
\item $d(\Sigma_{\Gamma}) \leq |\Sigma_{N} \cap \beta^*|$.
\item $|\Sigma_{\Gamma} \cap \beta^*|\leq 5$.
\end{enumerate}

If $(1)$ holds, then $|\beta^*\cap \Sigma_N|\geq |\beta^*\cap \Sigma_{\Gamma}|$ and $|S\cap \beta^*|+2|D\cap \beta^*|=|\Sigma_N\cap \beta^*|\geq k$, a contradiction, since $|S\cap \beta^*|+2|D\cap \beta^*|< k$. If $(2)$ holds, $d(\Sigma_{\Gamma}) \leq |S\cap \beta^*|+2|D\cap \beta^*|<k$, where this distance of $\Sigma_{\Gamma}$ is being measured as a bicompressible surface in $(N_c, \beta^* \cap N_c)$. However, by the previous claim, we have already established that the distance of $\Sigma_{\Gamma}$ as a bicompressible surface in $(N_c, \beta^* \cap N_c)$ is greater than or equal to the distance of $\Sigma_{\Gamma}$ as a bicompressible surface in $(S^3,\beta^*)$. This provides a contradiction, since the distance of $\Sigma_{\Gamma}$ as a bicompressible surface in $(S^3,\beta^*)$ is greater than or equal to $k$. Finally, $(3)$ does not hold since $|\Sigma_N \cap \beta^*|\geq 6$.

\textbf{Case 3:} $\Sigma_\Gamma^* $ is empty.

In this case, $\partial N_c$ is a collection of twice punctured spheres and, since $\beta^*$ is prime, $\Sigma_{\Gamma}$ and $\Sigma_N$ are  bridge spheres for $(S^3,\beta^*)$. In this case, we arrive at a contradiction in a similar manner as in Case 2, but in this case we set $M=N_c=S^3$. 

Since in all cases we reach a contradiction, $\Gamma$ is not contained in $N$.
\end{proof}

Since $\Gamma$ can be isotoped into $H_1^*$ or $H_2^*$, we arrive at a contradiction in exactly the same manner as in \cite{Pardon}: Recall that
\[
|S\cap \beta^*|+2|D\cap \beta^*|\leq 20(1+\frac{1}{\epsilon})\delta(\beta^*)<k.
\]
Thus, $|\bdd H_i^*\cap \beta^*|\leq k$. But $H_i^*$ is strictly contained in a congruent image of $Box(2^{-1/3}(1 + \epsilon) +\frac{\epsilon}{2} )$ so for $\Delta=1-2^{-1/3}(1 + \epsilon) -\frac{\epsilon}{2}$, we have $(1-\Delta)\mc{R}\subset \mc{R}$. Moreover, if $\epsilon =\frac{1}{7}$, then $0<(1-\Delta)<1$. This is impossible since a spine of a bridge sphere of a non-trivial knot cannot be isotoped to lie in an arbitrarily small region in $\mathbb{R}^3$. Hence,

$$ 160\delta(\beta^*)\geq k = \min(d(\Sigma), |\Sigma \cap \beta|).$$

\section{Comparing lower bounds}\label{sec:altJMknots}

In this section we show that the bound on distortion provided by Theorem \ref{thm:main} can be arbitrarily better than Pardon's lower bound which is a function of the representativity of a knot, a quantity first defined by Ozawa \cite{Oz}. Let $\beta$ be curve in $S^3$ embedded in a surface $F\subset S^3$ and let $S$ be the set of all isotopy classes of essential simple closed loops in $F$. Let $i:S \times S \rightarrow \mathbb{Z}_{\geq0}$ denote the minimum geometric intersection number. The \defn{representativity of $\beta$ with respect to $F$}, $I(F, \beta)$ is given by $min_{\alpha \in U} i(\alpha, \beta)$ where $U$ is the set of all nontrivial isotopy classes of loops $\alpha \in F$ which bound a PL embedded disk whose interior is disjoint from $F$. The \defn{representativity of a knot} $K$, $I(K)$ is given by $max_{F \in\mathcal{F}}I(F,K)$ where $\mathcal{F}$ is the set of all embedded surfaces containing $K$.

\begin{theorem}\label{the:Pardon}[\cite{Pardon}]
Let $F \subset \mathbb{R}^3$ be a PL embedded closed surface of genus $g \geq 1$. Let $\beta$ be an isotopy class of essential simple closed loops in $F$, and let $K_{\beta}$ denote the corresponding knot in $\mathbb{R}^3$. Then we have:

$$\delta(K_{\beta}) \geq \frac{1}{160} I(F, \beta)$$

\end{theorem}

In other words, $$\delta(K) \geq \frac{1}{160} I(K)$$
 Key to our construction is a recent result of Kindred.

\begin{theorem}\label{thm:kindred}[{\cite{Kindred} Main Theorem}]
Every non-split, non-trivial alternating link L has representativity $I(L) = 2$.
\end{theorem}

It remains to exhibit a sequence of non-trivial alternating knots with the desired bridge numbers and distances.

In \cite{JM} Johnson and Moriah showed that:

\begin{theorem}\label{thm:johnsonmoriah}[{\cite{JM}, Theorem 1.2 and Corollary 1.3}] \label{thm:JM} If $K$ is a $b$-bridge link type in $S^3$ with a highly twisted $n$-row, $2b$-plat projection for $b \geq 3$ and $n\geq 4b(b-2)$ then $d(K) = \lceil \frac{n}{(2(b-2))}\rceil$.

\end{theorem}

Figure \ref{fig:JMfig} depicts a knot of this type. Each box represents an alternating twist region with at least 3 crossings.  Note also that if the handedness of the twists alternates in alternating rows of twist regions, the resulting link is alternating.

\begin{figure}[ht]
\labellist \small\hair 2pt
\pinlabel {$o^+$} [b] at 37 27
\pinlabel {$o^+$} [b] at 86 27
\pinlabel {$o^+$} [b] at 249 27
\pinlabel {$e^+$} [b] at 37 312
\pinlabel {$e^+$} [b] at 86 312
\pinlabel {$e^+$} [b] at 249 312
\pinlabel {$e^+$} [b] at 37 200
\pinlabel {$e^+$} [b] at 86 200
\pinlabel {$e^+$} [b] at 249 200
\pinlabel {$e^-$} [b] at 13 83
\pinlabel {$e^-$} [b] at 61 83
\pinlabel {$e^-$} [b] at 111 83
\pinlabel {$e^-$} [b] at 224 83
\pinlabel {$e^-$} [b] at 273 83
\pinlabel {$e^-$} [b] at 13 256
\pinlabel {$e^-$} [b] at 61 256
\pinlabel {$e^-$} [b] at 111 256
\pinlabel {$e^-$} [b] at 224 256
\pinlabel {$e^-$} [b] at 273 256
\endlabellist
\includegraphics[scale=.6]{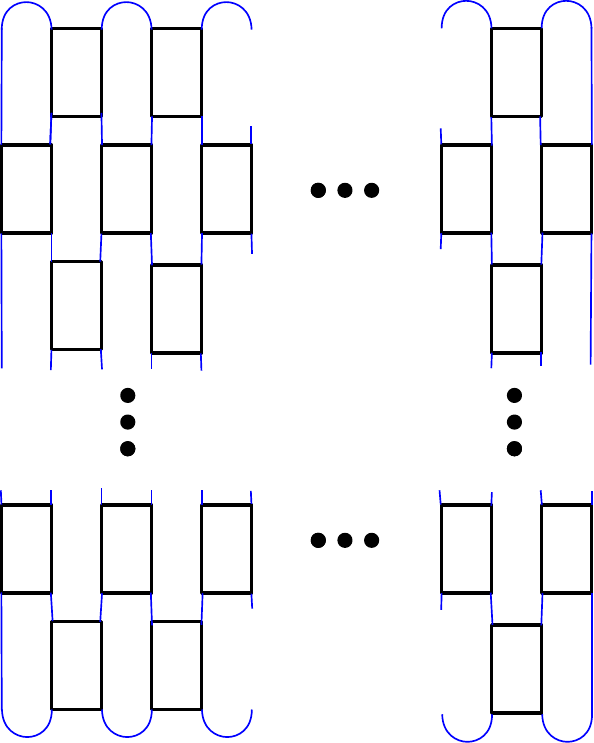}
\caption{An example of an alternating Johnson-Moriah knot. Each box represents an alternating twist region in the diagram. The label of $o$ indicates an odd number of crossings, the label of $e$ represents an even number of crossings and the super scripts of $+$ and $-$ represent right-handed and left-handed twisting respectively. Each region must contain at least 3 crossings.}
\label{fig:JMfig}
\end{figure}

\begin{theorem}\label{unboundedrep2}
There exists a sequence of knot types $\{K_i\}$ such that $lim_{i\rightarrow \infty} b(K_i)=lim_{i\rightarrow \infty} d(K_i)=\infty$ yet $I(K_i)= 2$ for all $i$.
\end{theorem}

\begin{proof}
Let $\{K_i\}_{i\geq 3}$ be a sequence of $i$-bridge knot types as in Figure \ref{fig:JMfig} where $K_i$ has a plat projection with $4i(i-2)$-rows and at least three half-twists in each twist region. As above, we can specify the handedness of the twisting to guarantee the links are alternating. Furthermore, to ensure that each $K_i$ is actually a knot, we additionally require that the last row of alternating twist regions contains an odd number of crossings and all other twist regions contain an even number of crossings. By Theorem \ref{thm:johnsonmoriah}, $d(K_i)=2i$. By Theorem \ref{tangdist1}, these knots have bridge number $i$ as required. Finally, by Theorem \ref{thm:kindred}, each of these knots has representativity equal to two.

\end{proof}

\section{An upper bound}
In this section, we give an upper bound for the distortion of the family of knots constructed by Johnson and Moriah.

\begin{theorem} \label{thm:upperbound}
Suppose $K$ has an $n$-row, $2b$-plat projection where each twist region has at least 3 half-twists and so that $b \geq 3$ and $n\geq 4b(b-2)$. Then $$\delta(K)\leq \mathcal{C}b^2d$$ where $d$ is the distance of the knot $K$ and $\mathcal{C}$ is a constant that only depends on the number of half-twists in each twist region.
\end{theorem}

\begin{proof}
We wish to construct an embedding, $\gamma$, of $K$ with a projection of the form of Figure \ref{fig:JMfig}, where each twist box is the projection of a cylinder in $\mathbb{R}^3$ of height $1$ and diameter $1$ and every cylinder touches each of the cylinders diagonally above and below it in a single point. Note that the number of rows, $n$, is always odd in this construction. Let $t_i$ be the number of half-twists in cylinder $i$ and assume each strand of the knot in a cylinder is a helix lying in the surface. We will think of $\gamma$ as decomposed into the following types of arcs: $2b$ bridge arcs, each of length $\pi/2$; $n+1$ vertical arcs (on the left and right edges of the diagram), each of length $1$; and $bn-\frac{n+1}{2}$ pairs of twisting arcs, where each arc in the pair has length $l_i$. The lengths $l_i$ can be easily computed in terms of the number of twists $t_i$ in the cylinder where the arc is located but that is not necessary for our purposes. Let $t=max\{t_i\}$ and let $l=max\{l_i\}$.

Consider a pair of points $p,q$ on $\gamma$.

\textbf{Claim 1:} If $p,q$ lie on the same arc, then $\frac{d_{\gamma}(p,q)}{d(p,q)} \leq 2\pi t$.

\textbf{Case A:} The arc that $p$ and $q$ lie on is a vertical arc.

In this case $\frac{d_{\gamma}(p,q)}{d(p,q)}=1$. If $p,q$ lie on a bridge arc, then $\frac{d_{\gamma}(p,q)}{d(p,q)}\leq \pi/2$.

\textbf{Case B:} The arc that $p$ and $q$ lie on is a twisting arc.
In this case $p$ and $q$ are points on a helix parametrized by $\displaystyle (r \cos \frac{x}{2}, r \sin \frac{x}{2}, \frac{x}{2 \pi t_i})$ where $0 \leq x \leq 2\pi t_i$ and $r=1/2$. Note that in this case $t_i$ is an integer but the proof holds for any real number for $t_i$. Let $\sigma =\frac{1}{2 \pi t_i}$. Then $p= (r \cos \frac{x_p}{2}, r \sin \frac{x_p}{2}, \sigma x_p)$ for some $x_p$ and $q= (r \cos \frac{x_q}{2}, r \sin \frac{x_q}{2}, \sigma x_q)$ for some $x_q$. Without loss of generality, $x_q > x_p$. Then the arc length from $p$ to $q$ is $(x_q-x_p)\sqrt{\frac{r^2}{4}+ \sigma^2}$. It follows that

\begin{align*}
\frac{d_{\gamma}(p,q)}{d(p,q)}=&\frac{(x_q-x_p)\sqrt{\frac{r^2}{4}+ \sigma^2}}{\sqrt{(r\cos \frac{x_q}{2}-r\cos \frac{x_p}{2})^2+(r\sin \frac{x_q}{2}-r\sin \frac{x_p}{2})^2+\sigma^2(x_q-x_p)^2}}\\
\leq &\frac{(x_q-x_p)\sqrt{\frac{r^2}{4}+ \sigma^2}}{\sqrt{\sigma^2(x_q-x_p)^2}} \\
=& \sqrt {\frac{\frac{r^2}{4}+ \sigma^2}{\sigma^2}} =\sqrt{\frac{1}{4}\pi^2t_i^2+1} \leq \pi t
\end{align*}
\qed

\textbf{Claim 2:} If $p,q$ lie on adjacent arcs then $\frac{d_{\gamma}(p,q)}{d(p,q)} \leq 2\pi t$.

If $p$ is in a twist region and $q$ is on an adjacent vertical or bridge arc, then let $u$ be the point in common of the two arcs. Note that $d(p,q) \geq max \{d(p,u), d(q,u)\}$. Assume that $d_{\gamma}(p,u) \geq d_{\gamma}(u,q)$. The other case is analogous. Then $$\frac{d_{\gamma}(p,q)}{d(p,q)}= \frac{d_{\gamma}(p,u)+d_{\gamma}(u,q)}{d(p,q)} \leq\frac{2d_{\gamma}(p,u)}{d(p,u)} \leq 2\pi t.$$ The last inequality follows by Claim 1.

Suppose $p$ and $q$ are on adjacent twisting arcs with twisting numbers $t_p$ and $t_q$ and let $u$ be the point in common of these arcs. Place a coordinate system with a center at $u$ so that the twisting arc containing $p$ has negative $x$ and $y$ coordinates and the arc containing $q$ has positive $x$ and $y$ coordinates. Consider the transformation which fixes the arc containing $p$ and sends each point $(x, y, z)$ of the arc containing $q$ to $(-x, y, z)$. Thus after the transformation the cylinder containing the point $q$ is stacked on top of the cylinder containing the point $p$. Let $q'$ be the image of $q$ under this transformation. The points $p$ and $q'$ are now on the same arc which is composed of two helical pieces. Note that $d_\gamma(p,q)=d_\gamma(p, q')$ and $d(p,q)\geq d(p, q')$. 

Consider a helical path $\gamma'$ from $p$ to $q$ of length max$\{ d_\gamma(p, u) , d_\gamma(u, q') \} \leq l(\gamma')\leq d_\gamma(p,q')$. We can obtain this path by unfolding the two cylinders to rectangles so that the diagonal of each rectangle has the same length as the corresponding twisting arc. See Figure \ref{fig:rectangle}. We can then use basic properties of triangles.

\begin{figure}[ht]
\labellist \small\hair 2pt
\pinlabel {$t_p$} [b] at 117 587
\pinlabel {$t_q$} [b] at 287 402
\pinlabel {$p$} [b] at 137 542
\pinlabel {$q$} [b] at 282 472
\pinlabel {$u$} [b] at 217 514
\pinlabel {$1$} [b] at 7 534
\pinlabel {$1$} [b] at 362 465
\pinlabel {$\gamma'$} [b] at 187 474
\endlabellist
\includegraphics[scale=.6]{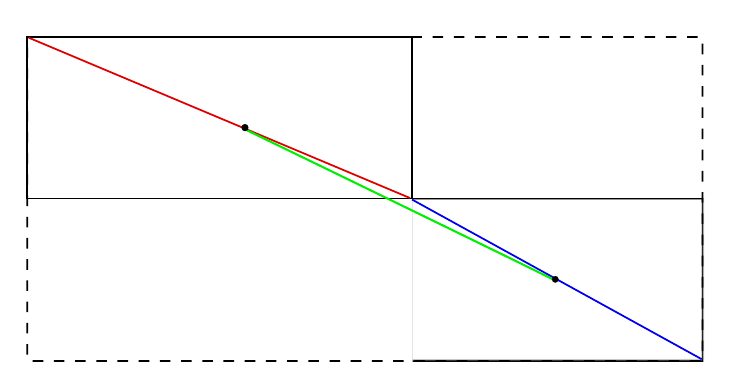}
\caption{}
\label{fig:rectangle}
\end{figure}

By Claim 1, Case B, it follows that:

$$\frac{d_{\gamma}(p,q)}{d(p,q)}\leq \frac{d_{\gamma}(p,q')}{d(p,q')} \leq \frac{2d_{\gamma'}(p,q')}{d(p,q')} \leq 2\pi t.$$

\qed

Suppose then that $p$ and $q$ are not on the same or adjacent arcs. Then there is a positive constant $\alpha$ which is independent of $b$ and $n$ so that $d(p,q) \geq \alpha$. This constant can be easily computed and it equals the minimum of the Euclidean distance between two arcs in the same twist region, but this is not needed for our proof.

Next we compute the maximum of $d_{\gamma}(p,q)$ over all pairs of points in the knot, i.e. half the length of the knot. As noted earlier, there are a total of $\displaystyle bn-\frac{n+1}{2}$ cylinders. Each cylinder contains $2$ strands with length at most $l$. There are also $2b$ bridge arcs each of length $\pi/2$ and a total of $n+1$ vertical segments each of length $1$. Thus $$d_{\gamma}(p,q)\leq \frac{(2bn-n-1)(l)+b\pi+n+1}{2}\leq bn(l+1).$$

From Theorem  \ref{thm:JM} we can conclude that $n \leq 2d(b-2)$ where $d=d(K)$. So $d_{\gamma}(p,q)\leq 2db(b-2)(l+1)\leq2db^2(l+1)\leq 4db^{2}l$.

We conclude that
\[
\delta(\gamma)=\sup\limits_{p,q\in \gamma} \frac{d_{\gamma}(p,q)}{d(p,q)}\leq\frac{4b^2dl}{\alpha}=\mathcal{C}b^2d\]
 where the  constant $\mathcal{C}$ only depends on the maximal number of half-twists over all twist regions and is independent of $d(K)$ and $b(K)$.
\end{proof}
%
%An immediate consequence is that we have an infinite family of knots whose distortion is bounded above and below by polynomials in their bridge number and the cube root of their crossing number.

\begin{corollary}\label{b3cor}
If $K_b$ is a knot in the Johnson-Moriah family with a $2b$ plat projection with $n = 4b(b-2)+1$ rows, then $\kappa_0 b \leq \delta(K_b) \leq \kappa_t b^3$
where $\kappa_0$ is a positive constant and $\kappa_t$ depends only on the maximum number, $t$, of twists in each twist box.

\end{corollary}
\begin{proof} Observe that by Theorem \ref{thm:johnsonmoriah} if $n = 4b(b-2)+1$ then $d(K_b) = 2b+1$. The desired bound then follows from Theorem \ref{thm:upperbound} and Theorem \ref{thm:main}.

\end{proof}

\begin{corollary}
If $K_b$ is an alternating knot in the Johnson-Moriah family with a $2b$ plat projection with $n = 4b(b-2)+1$ rows and three or four half-twists in every twist region, then $\kappa_0 (c(K_b))^{\frac{1}{3}} \leq \delta(K_b) \leq \kappa_4 c(K_b)$ where $\kappa_0$ and $\kappa_4$ are positive constants and $c(K_b)$ is the crossing number of $K_b$.

\end{corollary}

\begin{proof}
As in the proof of Theorem \ref{unboundedrep2}, we choose $K_b$ so that the diagram corresponding to Figure \ref{fig:JMfig} is reduced and alternating by specifying the handedness of the twisting in each twist region. Since there are at most four crossings in every twist region and every reduced, alternating diagram achieves crossing number \cite{K87, M87, Th87}, then the crossing number of $K_b$ is at most four times the number of twist regions. Since there are $bn-\frac{n+1}{2}=4b^3-10b^2+9b-1$ twist regions, then $$3(4b^3-10b^2+5b-1)\leq c(K_b) \leq 4(4b^3-10b^2+5b-1)$$.  Since $b^3\leq 4b^3-10b^2+5b-1\leq 4b^3$ for $b\geq 3$, it follows that $$3b^3\leq c(K_b) \leq 16b^3$$ for all knots under consideration. By Corollary \ref{b3cor}, as $t = 4$ there exist positive constants $\kappa_0$ and $\kappa_4$ such that $\kappa_0 (c(K_b))^{\frac{1}{3}} \leq \delta(K_b) \leq \kappa_4 c(K_b)$.
\end{proof}

\end{document}